\documentclass[12pt]{amsart}
\usepackage{amsmath}
\usepackage{amsxtra}
\usepackage{amscd}
\usepackage{amsthm}
\usepackage{amsfonts}
\usepackage{amssymb}
\usepackage{eucal}

\usepackage{mathabx}

\usepackage{xcolor}

\input prepictex
\input pictex
\input postpictex
\usepackage{mathptm}

\usepackage{verbatim}

\newtheorem{claim}{\indent Claim}
\newtheorem{theorem}{Theorem}[section]
\newtheorem*{theorem*}{Theorem}
\newtheorem{lemma}[theorem]{Lemma}
\newtheorem{proposition}[theorem]{Proposition}
\newtheorem*{proposition*}{Proposition}

\newtheorem{definition}[theorem]{Definition}
\newtheorem{remark}[theorem]{Remark}

\numberwithin{equation}{section}

\newcommand{\Z}{{\mathbb Z}}

\newcommand{\C}{{\mathbb C}}

\newcommand{\fg}{{\mathfrak g}}
\newcommand{\fh}{{\mathfrak h}}
\newcommand{\fn}{{\mathfrak n}}

\newcommand{\U}{{\rm U}}

\def\a{\alpha}
\def\b{\beta}
\def\d{\delta}
\def\D{\Delta}

\def\l{\lambda}

\def\es{\epsilon}

\begin{document}

\title[Integrable  representations of affine superalgebras]
{Integrable representations of \\ affine $A(m, n)$ and $C(m)$ superalgebras}

\author{Yuezhu Wu}
\author{R. B. Zhang}
\address[Wu]{School of Mathematics and Statistics, Changshu Institute of Technology, Changshu, Jiangsu, China}
\address[Wu, Zhang]{School of Mathematics and Statistics, University of Sydney, Sydney, NSW 2006, Australia}
\address{School of Mathematical Sciences, University of Science and Technology of China, Hefei, China}
\email{yuezhuwu@maths.usyd.edu.au}
\email{ruibin.zhang@sydney.edu.au}

\begin{abstract}
Rao and Zhao classified the irreducible integrable modules with finite dimensional
weight spaces for the untwisted affine superalgebras which are not $\hat{A}(m,n)$ ($m\ne n$) or $\hat{C}(m)$.
Here we treat the latter affine superalgebras to complete the classification.
The problem boils down to classifying the irreducible zero-level integrable modules with finite dimensional
weight spaces for these affine superalgebras, which is solved in this paper.
We note in particular that such modules for $\hat{A}(m,n)$ ($m\ne n$) and $\hat{C}(m)$ must be of highest weight type,
but are not necessarily loop modules. This is in sharp contrast to the cases of ordinary affine algebras and the other types
of affine superalgebras.

\end{abstract}
\maketitle


\section{Introduction}\label{sect:intro}

A finite dimensional simple Lie superalgebra $\fg$ over the field $\C$ of complex numbers
is called basic classical  \cite{k77} if its even subalgebra $\fg_{\bar0}$ is reductive and
$\fg$ carries an even non-degenerate supersymmetric invariant bilinear form $(\cdot\, | \, \cdot )$.
The full list of such simple Lie superalgebras can be found in \cite{k77}.
Fix a simple basic classical Lie superalgebra $\fg$ and let $L=\C[t,t^{-1}]$ be the algebra of Laurent polynomials in the
indeterminate $t$.  The untwisted affine superalgebra $G$ associated to $\fg$ \cite{JZ} is
$$G=\fg\otimes L\oplus \C c\oplus \C d$$
with the commutation relations defined as follows.  For any $X\in\fg$ and $m\in\Z$,
we denote $X(m)=X\otimes t^m$. Then for all  $a, b\in\fg$ and $m, n\in\Z$,
\[
\begin{aligned}
&{[}c,G]=0,  \qquad [d,a(m)]=m a(m), \\
&{[}a(m),b(n)]=[a,b](m+n)+ c m (a\,|\,b) \d_{m+n, 0}.
\end{aligned}
\]
We shall also denote the affine superalgebra $G$ by $\hat\fg$ following the convention of \cite{k83}.
Note that the even subalgebra of $G$ is $G_{\bar 0}=\mathfrak{g}_{\bar 0}\otimes L\oplus \C c \oplus \C d$,
which is the affine algebra of $\fg_{\bar0}$.

Let $H=\mathfrak{h}\oplus \C c \oplus \C d$, where $\fh$ is a Cartan subalgebra of $\fg$.
If $V$ is a $\Z_2$-graded $G$-module,  we denote by $V_\lambda$ the weight space of  $V$ with
weight $\lambda\in H^*$.  The module  is called integrable if (i) $V=\oplus_{\l\in H^*}V_\l$ and (ii) when restricted to a
$G_{\bar 0}$-module, $V$ is integrable in the usual sense (see \cite{ch86} and \cite[\S 3.6]{k83}).

The irreducible integrable modules with finite dimensional weight spaces for affine algebras associated
with finite dimensional simple Lie algebras were
classified by Chari \cite{ch86}, who proved that such modules comprise of  irreducible integrable highest weight
modules, irreducible integrable lowest weight modules and loop modules.
The irreducible integrable modules for affine superalgebras
were investigated systematically by S. Rao and K. Zhao
and others (see \cite{rz04} and references therein).

Let $V$ be an irreducible integrable module for $G$ with finite dimensional weight spaces.
Since $c$ is also the central extension of $G_{\bar0}$,  it is known \cite{k83} that $c$ must act on $V$ by an integer,
which is called the level of $V$.
It has long been known \cite{JZ, kw01} that the affine superalgebras $\hat{A}(m,n)$ ($m\ge 1,\,n\ge 1$),
$\hat{B}(m,n)$ ($m\ge 1, n\ge 1$),  $\hat{D}(m,n)$  ($m\ge 2, n\ge 1)$,  $\hat{D}(2,1,\a)$,  $\hat{F}(4)$ and $\hat{G}(3)$ do not
admit any integrable modules of nonzero level.
In the cases when $G$ is $\hat{A}(0,n) $($n\ge 1$), $\hat{B}(0,n)$ ($n\ge 1$) and $\hat{C}(m)$ ($m\ge 3$),
if the level is a positive (resp. negative) integer, then by results of \cite{rz04}, $V$ is a highest (resp. lowest) weight module
with respect to the Borel subalgebra $\mathfrak{b}\oplus\fg\otimes t\C[t]\oplus\C c\oplus\C d$ of $G$, where $\mathfrak{b}$ is a Borel subalgeba of $\fg$.
Therefore, the classification of irreducible integrable modules reduces to the classification of
those of zero-level.

It is proved in \cite{rz04}  that any irreducible zero-level integrable module
with finite dimensional weight spaces for the affine superalgebra $G$
is an loop module provided that $G$ is not  $\hat{A}(m,n)$ ($m\ne n$) or $\hat{C}(m)$.
The method used to prove this in \cite{rz04} is an adaption to the affine superalgebra context of the method developed by Chari \cite{ch86} for ordinary (i.e., non-super) affine algebras.
Semi-simplicity of $\fg_{\bar0}$ was used in a crucial way in proving that irreducible integrable $G$-modules were of highest weight
type. This condition is not met in the cases of $A(m,n)$ ($m\ne n$) and $C(m)$.
This is a main reason why the method of \cite{rz04} failed to produce a complete classification.

The aim of this paper is to complete the classification of irreducible integrable modules for untwisted affine superalgebras
started by Rao and Zhao by treating the affine superalgebras $\hat{A}(m,n)$ ($m\ne n$) and $\hat{C}(m)$.
This is achieved in Theorem \ref{thm:main}.
The irreducible zero-level integrable modules with finite dimensional weight spaces
for $\hat{A}(m,n)$ ($m\ne n$) and $\hat{C}(m)$
are all highest weight modules with respect to the triangular decomposition
\eqref{eq:triangular} (see Theorem \ref{thm:hw}), but in sharp contrast to Theorem \ref{thm:ev}
for ordinary affine algebras and the other types of affine superalgebras,
such modules are not necessarily loop modules. The necessary and sufficient conditions for
a simple highest weight module to be a loop module is given in
Lemma \ref{lem:ev-condition} and Remark \ref{rem:final}.

Let us briefly describe the content of this paper.
In Section \ref{sect:ev-mod} we recall the construction of loop modules for affine superalgebras.
In Section  \ref{sect:hw-mod} we prove that any irreducible zero-level integrable module with finite dimensional weight spaces
for $\hat{A}(m,n)$ ($m\ne n$) and $\hat{C}(m)$
must be a highest weight module.  The result is given in Theorem \ref{thm:hw}.
In Section \ref{sect:new} we construct irreducible  integrable modules (see Definition  \ref{def:V}) for
these affine superalgebras, which include the irreducible loop modules as a special case.
In the last section we prove the main result of this paper, that is, Theorem \ref{thm:main},
which states that any irreducible zero-level integrable module with finite dimensional weight spaces
for $\hat{A}(m,n)$ ($m\ne n$) and $\hat{C}(m)$ is one of the modules given in Definition \ref{def:V}.

We point out that the method used in this paper to classify irreducible integrable modules
is very much inspired by the work \cite{r11} of Rao on finite dimensional
modules for multi-loop superalgebras. It is very different from that of \cite{rz04}.
As far as we are aware, the method of
\cite{rz04} has not been improved to deal with
$\hat{A}(m,n)$ ($m\ne n$) and $\hat{C}(m)$.

\section{Loop modules}\label{sect:ev-mod}
We recall results from \cite{rz04,r04} on irreducible integral modules for affine superalgebras, which will be needed later.

\subsection{Highest weight modules}

Given a simple basic classical Lie superalgebra $\fg$, we let $\fg=\mathfrak{n}^-\oplus \mathfrak{h}\oplus \mathfrak{n}^+$
be the triangular decomposition with $\mathfrak{b}=\mathfrak{n}^+\oplus \mathfrak{h}$ being a distinguished Borel subalgebra and $\mathfrak{h}$ a Cartan subalgebra.

Denote $L(\fg)=\fg\otimes L$ and $L(\mathfrak{h})=\mathfrak{h}\otimes L$ with $L=\C[t,t^{-1}]$.  Let
\[
 G^{\pm }=\mathfrak{n}^{\pm }\otimes L,  \quad
T_0=L(\mathfrak{h})\oplus\C c, \quad T=T_0\oplus\C d.
\]
The affine superalgebra $G=L(\fg)\oplus \C c\oplus \C d$ associated with $\fg$ contains the subalgebra $G'=L(\fg)\oplus \C c$. We have
the following  triangular decompositions for $G$ and $G'$:
\begin{equation}\label{eq:triangular}
G=G^-\oplus T\oplus G^+, \quad G'=G^-\oplus T_0\oplus G^+.
\end{equation}

We shall deal only with elements $\l\in H^*$ such that $\l(c)=0$, i.e.,
$\l\in (\mathfrak{h}\oplus \C d)^*.$
A module $V$ of $G$ (resp. $G'$) is called a {\it highest
weight module} if there exists a weight vector $v\in V$ with respect
to $\mathfrak{h}\oplus\C c\oplus\C d$ (resp. $\mathfrak{h}\oplus\C c$) such that (1)
$\U(G)v=V$ (resp. $\U(G')v=V$), (2) $G^+v=0$, and (3)
$\U(T)v$ (resp.  $\U(T_0)v$) is an irreducible
$T$-module (resp.  $T_0$-module).  The vector $v$ is called a highest
weight vector of $V$.

Let $\tilde\varphi : \U(T_0)\rightarrow L$ be a $\Z$-graded algebra homomorphism
such that $\tilde\varphi(c)=0$ and $\tilde\varphi|_\mathfrak{h}\in
\mathfrak{h}^*$. Then for any given $b\in\C$, we can turn $L$ into
a $T$-module via $\tilde\varphi$ defined for all $f\in L$ by
\begin{eqnarray}\label{eq:T-mod}
d f = \left(t\frac{d }{d t} + b\right)f,\quad c f=0, \quad h(m) f=\tilde\varphi(h(m)) f, \quad h(m)\in L(\mathfrak{h}).
\end{eqnarray}
We write $\varphi=(\tilde\varphi, b)$ and denote by $L_\varphi$ the image of  $\tilde\varphi$
regarded as a $\Z$-graded $T$-submodule.
It was shown in \cite[\S 3]{ch86}  that if $L_\varphi$ is
a simple $T$-module, it must be $L_0:=\C$ or a Laurent subring
$L_r:=\C[t^r,t^{-r}]$ for some integer $r> 0$.

Assume that  $L_\varphi$ is a simple $T$-module. We extend $L_\varphi$ to a module over $B:=G^+\oplus T$
with $G^+$ acting trivially, and construct the induced $G$-module
\begin{eqnarray}\label{eq:Verma}
M(\varphi)=\U(G)\otimes_{\U(B)}L_\varphi.
\end{eqnarray}
This has a unique irreducible quotient, which we denote by ${\hat V}(\varphi)$.  Then every
irreducible highest weight $G$-module is isomorphic to some ${\hat V}(\varphi)$.

\begin{remark}
Let $\C_a$ denote the $1$-dimensional $G$-module with $G'$ acting trivially and $d$ acting by multiplication by $a\in\C$,  then $\hat V(\tilde\varphi, b+a)\cong \hat V(\tilde\varphi, b)\otimes_\C\C_a$.
\end{remark}

Define the evaluation map $S:L \rightarrow \C$, $t\mapsto 1$ and set $\psi=S\circ \tilde\varphi: \U(T_0)\rightarrow
\C$. Let $\U(T_0)$ act on the one dimensional vector
space $\C_\psi=\C$ by $\psi$.  We extend $\C_\psi$ to a module over $B':=G^+\oplus
T_0$ by letting $G^+$ act trivially.
Construct the induced $G'$-module
\[
M(\psi)=\U(G')\otimes_{\U(B')}\C_\psi,
\]
which also has a unique simple quotient $V(\psi)$.

Form the vector space $V(\psi)\otimes L$ and denote $w(s)=w\otimes t^s$
for any $w\in V(\psi)$ and $s\in \Z$.  We now turn $V(\psi)\otimes L$ into a
$G$-module by defining the action
\begin{eqnarray}\label{L-module}
\begin{aligned}
& c w(s)=0, \quad d w(s)=(s+b) w(s),\\
& x(m)w(s)=(x(m)w)(s+m), \quad x(m)\in L(\fg).
\end{aligned}
\end{eqnarray}

The following results due to Rao and Zhao \cite{rz04,r04} will be important later.  In \cite{rz04,r04}
only the $b=0$ case was stated, which implies the general case.

\begin{theorem}\label{iso}\cite{rz04,r04}
Let $\varphi$ and $\psi$ be as above. Assume that $L_\varphi\cong L_r$ is
an irreducible $T$-module. Let  $v$ be a highest weight vector of $V(\psi)$ and denote
$v(i)=v\otimes t^i$ for any $i\in \Z$. Then
\begin{enumerate}
\item
$V(\psi)\otimes L\cong \oplus_{i=0}^{r-1}\U(G)v(i)$ as
$G$-modules, where $\U(G)v(i)$ are irreducible $G$-submodules. Furthermore, $\U(G)v(0)\cong \hat V(\varphi)$.

\item $\hat V(\varphi)$ has finite dimensional weight spaces with respect
to $\mathfrak{h}\oplus \C d$ if and only if $V(\psi)$ has finite
dimensional weight spaces with respect to $\mathfrak{h}$.

\item
$V(\psi)$ has finite dimensional weight spaces if and only if $\psi$ factors through $\mathfrak{h}\otimes L/I$ for some co-finite ideal $I$ of $L$. In this case $(\fg\otimes I)V(\psi)=0.$
\end{enumerate}
\end{theorem}

We note that $\U(G)v(i)\cong V(\tilde\varphi, b+i)$. In the case $r=0$, the formula in part (1)  of the theorem should be understood as $V(\psi)\otimes L\cong \oplus_{i\in\Z}\U(G)v(i)$.

\begin{remark}\label{iso-semi}
Similar arguments as those in \cite{rz04,r04} can show that Theorem \ref{iso}
 still holds when $\fg$ is a  semi-simple Lie algebra.
\end{remark}

\begin{remark}\label{ideal} The co-finite ideal $I$ can be chosen to be generated by a polynomial $P(t)$ (see \cite{r04}).
By multiplying it by $t^m$  $(m\in\Z)$, we can also assume that $P(t)$ has non-zero roots.
\end{remark}

\subsection{Loop modules}
Now we recall the construction of  loop modules.
Denote by $V(\l)$ the irreducible highest weight $\fg$-module
with highest weight $\l\in\mathfrak{h}^*$.
Let $K$ be a positive integer, and fix a $K$-tuple $\underline{a}=(a_1,\dots ,a_K)$
of complex numbers, which are all distinct and non-zero.
Define a Lie superalgebra homomorphism
\begin{eqnarray} \label{eq:zeta}
\zeta:L(\fg)\rightarrow \fg_K=\underbrace{\fg\oplus \cdots \oplus \fg}_K, \quad
                    \zeta(x\otimes t^m)=(a_1^mx,\cdots,a_K^mx),
\end{eqnarray}
for all $x\in\fg$ and $m\in\Z$. Then $\zeta$ is surjective under the given conditions for $\underline{a}$.

\begin{remark}\label{rem:ideal}
Let $I$ be the ideal of $L$ generated by $P(t)=\prod\limits_{i=1}^K(t-a_i)$.
 It was shown in \cite{r04} that ${\rm ker}\zeta=\fg\otimes I$ and $\fg\otimes L/I\cong \fg_K$.
\end{remark}

Given irreducible $\fg$-modules $V(\l_1),\dots,V(\l_K)$ with integral dominant highest
weights $\l_1,\dots,\l_K$ respectively,  we let
\begin{eqnarray}\label{eq:ev-G'}
V(\underline{\l},\underline{a})= V(\l_1)\otimes V(\lambda_2)
\otimes \dots \otimes V(\lambda_K),
\end{eqnarray}
where  $\underline{\l}=(\l_1,\dots,\l_K)$. Then $V(\underline{\l},\underline{a})$
is an irreducible highest weight $L(\fg)$-module via $\zeta$, where irreducibility follows from the surjectivity of $\zeta$.
Such modules are called evaluation modules for $L(\fg)$.

Let $v_i$ be a highest weight vector of $V(\lambda_i)$
for each $i$. Then $v:=v_1\otimes v_2 \otimes \dots\otimes v_K$ is a highest weight vector of $V(\underline{\l},\underline{a})$
satisfying
\begin{eqnarray}\label{eq:hw-tensor}
(h\otimes t^m)v= \sum_{j=1}^Ka_j^m\l_j(h) v, \quad \forall h\in\mathfrak{h}.
\end{eqnarray}

Define an algebra homomorphism $\psi: \U(T_0)\rightarrow \C$ by $\psi(c)=0$ and
\[
\psi(h\otimes t^m)=\sum_{j=1}^Ka_j^m\l_j(h), \quad  \forall h\in\fh, \ m\in \Z.
\]
Then it follows from \eqref{eq:hw-tensor}  that $V(\psi)\cong V(\underline{\l},\underline{a})$
since $V(\psi)$ is determined by $\psi$.

Introduce the $\Z$-graded algebra homomorphism
$\tilde\varphi:\U(T_0)\rightarrow L$ defined by
\begin{eqnarray}\label{eq:phi}
\tilde\varphi(c)=0, \quad \tilde\varphi(h\otimes t^m)= \sum_{j=1}^Ka_j^m\l_j(h) t^m, \quad \forall  h\in\mathfrak{h}, \ m\in\Z.
\end{eqnarray}
Then $im(\tilde\varphi)$ is a simple $T$-module via \eqref{eq:T-mod} for any fixed $b\in\C$,
and there exists an integer $r\ge 0$ such that $im(\tilde\varphi)=L_r$ by \cite{cp86}.  [Note that
$r=0$ if and only if $\lambda_i=0$ for all $i$.]
It follows from Theorem \ref{iso}
that $V(\underline{\l},\underline{a})\otimes L=\oplus_{i=0}^{r-1} \hat V_i(\underline{\l},\underline{a})$, where
$\hat V_i(\underline{\l},\underline{a}):= \U(G)(v\otimes t^i)$. Note in particular that $\hat V_0(\underline{\l},\underline{a})=\hat V(\varphi)$ with $\varphi=(\tilde\varphi, b)$.

\begin{definition}
Call the $G$-modules $\hat V_i(\underline{\l},\underline{a})$ simple loop modules.
\end{definition}

If $r=1$,  then $\hat V(\underline{\l},\underline{a})\otimes L$ is simple.
When $r>1$,  assuming $\l_i\ne 0$ for all $i$, we have
$K=r S$ for some positive integer $S$, and there exists a permutation $\sigma$
such that $\sigma(\underline{\lambda})=(\lambda_{\sigma(1)}, \lambda_{\sigma(2)}, \dots, \lambda_{\sigma(K)})$
and $\sigma(\underline{a})= (a_{\sigma(1)}, a_{\sigma(2)}, \dots, a_{\sigma(K)})$ are respectively given by
\[
\begin{aligned}
\sigma(\underline{\lambda})&=(\underbrace{\mu_1, \mu_1, \dots, \mu_1}_r, \underbrace{\mu_2, \mu_2,\dots, \mu_2}_r, \dots, \underbrace{\mu_S, \mu_S, \dots, \mu_S}_r),\\
\sigma(\underline{a})&=(b_1, \omega b_1,\dots, \omega^{r-1} b_1, b_2, \omega b_2 \dots, \omega^{r-1} b_2, \dots, b_S, \omega b_S, \dots, \omega^{r-1}b_S),
\end{aligned}
\]
where $\omega=\exp\left(\frac{2\pi\sqrt{-1}}{r}\right)$,  and $b_1, \dots, b_S\in \C^*$ such that $b_i/b_j$ is not a power of $\omega$ if $i\ne j$.
The ideal $I\unlhd L$ in Theorem \ref{iso}(3) associated with $V(\psi)=V(\underline{\l},\underline{a})$ is generated by the polynomial
$
\prod_{s=1}^S \prod_{i=0}^{r-1}(t - \omega^i b_s)=\prod_{s=1}^S (t^r - b_s^r).
$
It immediately follows from the obvious fact $\sum_{i=0}^{r-1} \omega^i=0$ that $\psi(h\otimes t^m)=0$ in this case and hence $\varphi(h\otimes t^m)=0$  unless $r|m$.
The loop modules $\hat V_i(\underline{\l},\underline{a})$ when $r>1$
can be explicitly described as in \cite[\S 4]{cp86}, but their detailed structure will not play any significant role in the remainder of the present paper.

We have the following result.
\begin{theorem}\cite{ch86, cp86, rz04}\label{thm:ev}
Suppose $\fg$ is a simple Lie algebra or is a basic classical Lie superalgebra not of type $A(m,n)$ ($m\ne n$) or $C(m)$.
Then any irreducible zero-level integrable module for the associated affine (super)algebra $G$
with finite dimensional weight spaces is a simple loop module.
\end{theorem}
The proof of the theorem was given in \cite{ch86, cp86} when $\fg$ is a simple Lie algebra,
and in \cite{rz04} when $\fg$ is a basic classical Lie superalgebra.
\begin{remark}\label{semi}
Similar arguments as those in \cite{ch86, cp86, rz04} can show that Theorem
\ref{thm:ev} still holds for semi-simple Lie algebras.
\end{remark}

In the remainder of the paper, we classify the irreducible zero-level integrable modules
with finite dimensional weight spaces for the affine superalgebras $\hat{A}(m,n)$ ($m\ne n$) and $\hat{C}(m)$. This requires algebraic methods quite different from those used in \cite{ch86, cp86} and  \cite{rz04}.

\section{Highest weight modules for $\hat{A}(m,n)$ and $\hat{C}(m)$}\label{sect:hw-mod}
We will show in this section that irreducible zero-level integrable modules with finite dimensional weight spaces for the affine superalgebras
$\hat{A}(m,n)$ ($m\ne n$) and $\hat{C}(m)$ must be highest weight modules with respect to the triangular decomposition \eqref{eq:triangular}, see Theorem \ref{thm:hw}
for the precise statement.

Detailed structures of  the underlying
finite dimensional simple Lie superalgebras $A(m,n)$ and $C(m)$ will be
required, which we describe below.

\subsection{Lie superalgebras $\mathfrak{sl}(m,n)$ and $C(m)$}
Recall that $A(m, n)$ is
$\mathfrak{sl}(m+1, n+1)$ if $m\ne n$, and is $\mathfrak{sl}(n+1|n+1)/\C  I_{2n+2}$ if $m=n$. Also
$C(m)=\mathfrak{osp}(2|2m-2)$.
To simplify notation, we consider $\mathfrak{sl}(m, n)$ instead of $A(m, n)$ in this section.
\subsubsection{The Lie superalgebra $\mathfrak{sl}(m,n)$}
Let $V=V_{\bar {0}}\oplus V_{\bar{1}}$ be a $\Z_2$-graded vector space with ${\rm dim}V_{\bar{0}}=m$ and ${\rm dim}V_{\bar{1}}=n$.
Then the space of $\C$-linear endomorphisms ${\rm End }(V)$ on $V$ is also $\Z_2$-graded,
${\rm End }(V)=({\rm End }(V))_{\bar{0}}\oplus ({\rm End }(V))_{\bar{1}}$, with
$$({\rm End }(V))_{j}=\{f\in {\rm End }(V)|\,f(V_k)\subset V_{k+j} \mbox{ for all }k\in\Z_2\}.$$
The general linear Lie superalgebra $\mathfrak{gl}(m, n)$ is ${\rm End }(V)$ endowed with the following Lie super bracket
$$ [f,g]=f\cdot g-(-1)^{ij}f\cdot g,\quad f\in ({\rm End }(V))_{i},\  g\in ({\rm End }(V))_{j}.$$
By fixing bases for $V_{\bar{0}}$ and $V_{\bar{1}}$ we can write $X\in {\rm End }(V)$ as
$X=\begin{pmatrix}
A&B\\ C&D
\end{pmatrix},$
where $A$ is an $m\times m$ matrix, $B$ is an $m\times n$ matrix,
$C$ is an $n\times m$ matrix and $D$ is an $n\times n$ matrix.
Then $\begin{pmatrix}
A&0\\ 0&D
\end{pmatrix}$ is even and $\begin{pmatrix}
0&B\\ C&0
\end{pmatrix}$ is odd.
Denote by $E_{a b}$ the
$(m+n)\times(m+n)$-matrix unit, which has zero entries everywhere
except at the $(a, b)$ position where the entry is $1$. Then  ${\mathfrak{gl}}(m,n)$ has the
homogeneous basis $\{E_{a b}\, |\,1\le a,b\le m+n\}$  with $E_{a b}$
being even if $1\le a, b\le m$, or $m+1\le a, b\le m+n$, and odd
otherwise.
Let $\es_a$ ($a=1, 2, \dots, m+n$) be elements in the dual space of  $\tilde\fh:=\sum_{a=1}^{m+n} \C E_{ a a}$
such that $\epsilon_a(E_{b b})=\delta_{a b}$. There exists a non-degenerate bilinear form $(\cdot,\cdot):\tilde\fh^* \times \tilde\fh^*\mapsto \C$ such that  $(\es_a,\es_b)=(-1)^{[a]}\delta_{a b}$ where $[a]=0$ if
$a\le m$ and $1$ if $a>m$.

The special linear Lie superalgebra $\mathfrak{sl}(m,n)$ is the Lie sub-superalgebra of $\mathfrak{gl}(m, n)$ consisting of
elements  $X\in\mathfrak{gl}(m, n)$ such that ${\rm str}X=0$, where  the supertrace is defined for any
$X=\begin{pmatrix}
A&B\\ C&D
\end{pmatrix}$
by ${\rm str}X={\rm tr}A-{\rm tr}D$.
It is well known that $\fg=\mathfrak{sl}(m, n)$ is simple if $m\ne n$. However, if $m=n$, the identity matrix
belongs to $\fg$, which clearly spans an ideal.

Let $\fg_{0}=\left\{\begin{pmatrix}
A&0\\ 0&D
\end{pmatrix} \right\}$,
$\fg_{+1}=\left\{\begin{pmatrix}
0&B\\ 0&0
\end{pmatrix}\right\}$ and
$\fg_{-1}=\left\{\begin{pmatrix}
0&0\\ C&0
\end{pmatrix}\right\}
$.
Then $\fg_{\bar{0}}=\fg_0$ and $\fg_{\bar{1}}=\fg_{-1}\oplus\fg_{+1}.$ Note that $\fg_0$ is reductive and $\fg_0\cong\mathfrak{sl}(m)
\oplus\C z\oplus\mathfrak{sl}(n)$, where $\C z$ is the center of $\fg_{0}$. Also, $\fg$ admits a $\Z$-grading
$\fg=\fg_{-1} \oplus \fg_{0}\oplus \fg_{+1}$ with $\fg_{\pm1}$ satisfying
$[\fg_{+1}, \fg_{+1}]=0=[\fg_{-1}, \fg_{-1}]$.

Let $\mathfrak{h}$ be the standard Cartan subalgebra
consisting of the diagonal matrices in $\fg$. Denote $\d_j=\es_{m+j}$ for $1\le j\le n.$
The sets of  the positive even roots,  positive odd roots and simple roots are respectively given by
  \begin{eqnarray*}
  &&\D_{\bar{0}}^+=\{\es_i-\es_{j},\,\d_k-\d_l|\,1\le i<j\le m,\,1\le k<l\le n\},\\
  &&\D_{\bar{1}}^+=\{\es_i-\d_j|\, 1\le i\le m,\,1\le j\le n\},\\
  &&\Pi=\{\es_1-\es_2,\cdots,\es_{m-1}-\es_{m},\es_{m}-\d_1,\d_1-\d_2,\cdots,\d_{n-1}-\d_{n}\},
  \end{eqnarray*}
where $\Pi$ forms a basis of $\fh^*$. Set $\D^+=\D_{\bar{0}}^+\bigcup \D_{\bar{1}}^+$ and
$\D=\D^+\bigcup (-\D^+)$.  The root system can be encoded in the Dynkin diagram

\begin{center}
\begin{picture}(160, 20)(0, 0)
\put(10, 10){\circle{10}}
\put(15, 10){\line(1, 0){10}}
\put(25, 9){...}
\put(35, 10){\line(1, 0){10}}
\put(50, 10){\circle{10}}
\put(55, 10){\line(1, 0){20}}
{ \color{gray} \put(75, 10){\circle*{10}} }
\put(80, 10){\line(1, 0){20}}
\put(105, 10){\circle{10}}
\put(110, 10){\line(1, 0){10}}
\put(120, 9){...}
\put(130, 10){\line(1, 0){10}}
\put(145, 10){\circle{10}}
\put(151, 8){,}
\end{picture}
\end{center}
where the grey node corresponds to the odd simple root.

For any root $\a\in \D$,
we denote by $\fg_\a$ the corresponding root space.  Let $\mathfrak{n}_{\bar{i}}^\pm=\bigoplus_{\pm\a\in \D^+_{\bar{i}}}\fg_\a$ ($i=0, 1$) and $\mathfrak{n}^\pm=\mathfrak{n}_{\bar{0}}^\pm\oplus \mathfrak{n}_{\bar{1}}^\pm$,
then $\fg=\mathfrak{n}^-\oplus \mathfrak{h}\oplus \mathfrak{n}^+$. Note that
$\fn_{\bar 1}^{\pm}=\fg_{\pm 1}$.

\subsubsection{The Lie superalgebra $C(m)$}
The structure of $\fg=C(m)$ can be understood by regarding it as a Lie subalgebra of $\mathfrak{sl}(2,2m-2)$
that preserves a non-degenerate supersymmetric bilinear form.
To describe the root system of $\fg$,
we let $\mathfrak{h}$ be the Cartan subalgebra of $\fg$ contained in the distinguished Borel subalgebra.
Then $\fh^*$ has a basis
$\epsilon, \delta_1, \dots, \delta_{m-1}$ and is equipped with the bilinear form $(\cdot,\cdot):\mathfrak{h}^* \times \mathfrak{h}^*\mapsto \C$ such that $(\es,\es)=1$, $(\d_k,\d_l)=-\d_{k l}$ and  $(\es,\d_k)=0.$ The sets of  positive even roots,
positive odd roots and simple roots are respectively given by
  \begin{eqnarray*}
  &&\D_{\bar{0}}^+=\{\d_i-\d_j,\,2\d_1,\,2\d_j|\,1\le i<j\le m-1\},\\
  &&\D_{\bar{1}}^+=\{\es\pm \d_j|\, 1\le j\le m-1\},\\
  &&\Pi=\{\es-\d_1,\d_{1}-\d_{2},\d_{m-2}-\d_{m-1},2\d_{m-1}\}.
  \end{eqnarray*}
We denote $\D^+=\D_{\bar{0}}^+\bigcup \D_{\bar{1}}^+$ and $\D=\D^+\bigcup (-\D^+)$.   The Dynkin diagram of the root system is given by

\begin{center}
\begin{picture}(180, 20)(0, 0)
{\color{gray} \put(10, 10){\circle*{10}} }
\put(15, 10){\line(1, 0){20}}
\put(40, 10){\circle{10}}
\put(45, 10){\line(1, 0){10}}
\put(55, 9){...}
\put(65, 10){\line(1, 0){10}}
\put(80, 10){\circle{10}}
\put(85, 10){\line(1, 0){20}}
\put(110, 10){\circle{10}}
\put(115, 9){\line(1, 0){20}}
\put(115, 11){\line(1, 0){20}}
\put(115, 7){$<$}
\put(140, 10){\circle{10}}
\put(146, 8){.}
\end{picture}
\end{center}

We have $\fg=\mathfrak{n}^-\oplus \mathfrak{h}\oplus \mathfrak{n}^+$,
where $\mathfrak{n}^\pm=\mathfrak{n}_{\bar{0}}^\pm\oplus \mathfrak{n}_{\bar{1}}^\pm$
with $\mathfrak{n}_{\bar{i}}^\pm=\bigoplus_{\pm\a\in \D^+_{\bar{i}}}\fg_\a$ ($i=0, 1$).
The even subalgebra $\fg_{\bar{0}}=\mathfrak{n}_{\bar{0}}^-\oplus\fh\oplus\mathfrak{n}_{\bar{0}}^+$
is $\C z\oplus \mathfrak{sp}_{2m-2}$,  where $\C z$ is the center of $\fg_{\bar0}$.
From the root system one immediately sees that
$\fg$ admits a $\Z$-grading $\fg=\fg_{-1} \oplus \fg_{0}\oplus \fg_{+1}$
with $\fg_0=\fg_{\bar0}$ and $\fg_{\pm1}=\mathfrak{n}_{\bar{1}}^\pm$.

It is known that $C(m)\cong\mathfrak{sl}(2,1)$ if $m=2$. Thus we may assume that $m\ge 3$.

\subsection{Highest weight modules} Let $G$ be the affine superalgebra associated with
$C(m)$ or $\mathfrak{sl}(m,n)$ with $m\ne n$.
We retain notation of Section \ref{sect:ev-mod}, and set $G^{\pm }_{\bar{0}}=\mathfrak{n}_{\bar{0}}^{\pm }\otimes L$.
Define $\d\in H^*$ by setting $\d|_{\mathfrak{h}\oplus \C c}=0$ and $\d(d)=1.$
We write  $\l\le \mu$ for  $\l, \mu\in H^*$
if $(\mu-\l)|_\mathfrak{h}=\sum k_i\a_i$ with $k_i$ non-negative integers and $\a_i\in \Pi$.

Let $V$ be an irreducible zero-level integrable module for $G$ with finite dimensional weight spaces.
In this subsection we will show that $V$ has to be a highest weight module with respect to the triangular decomposition \eqref{eq:triangular} of $G$. This  will be  done in detail for $\widehat{sl}(m, n)$ only as the proof for $\hat{C}(m)$ is similar.

By the definition of integrable $G$-modules, $V$ is integrable over the even subalgebra $G_{\bar0}$ of $G$. It follows from
Chari's work \cite{ch86} that there is a non-zero weight vector $v\in V$ such that $G^+_{\bar{0}}v=0$.
Denote by $wt(v)$ the weight of $v$. Let $X$ be the subspace of $V$ spanned by the vectors
$E_{m,m+1}(k)E_{m,m+1}(-k)v$ for all $k\ge 0$, which is a subspace of $V_{wt(v)+2(\epsilon_m-\delta_1)}$,
thus $\dim X<\infty$. Therefore, there exists a finite positive integer $N$ such that
\[
X=span\{E_{m,m+1}(k)E_{m,m+1}(-k)v \mid 0< k<N\}.
\]
Thus  for any $r\in\Z$ we have
\begin{equation}\label{N}
E_{m,m+1}(r)E_{m,m+1}(-r)v=\sum\limits_{0< k<N}a^{(r)}_k E_{m,m+1}(k)E_{m,m+1}(-k)v,\quad a^{(r)}_k\in\C.
\end{equation}

Note that the elements $E_{m, m+1}(k)$ for all $k\in\Z$ anti-commute among themselves and satisfy
$E_{m, m+1}(k)^2=0$. Thus equation \eqref{mm+1-} in the lemma below immediately follows from \eqref{N}.
\begin{lemma}\label{lem:v-vector}
Let $V$ be an irreducible zero-level integrable $G$-module,  and
let $v\in V$ be a non-zero weight vector such that $G^+_{\bar{0}}v=0$.
Then the following relations hold for large  $k$:
\begin{equation}\label{mm+1-}
E_{m,m+1}(n_1)E_{m,m+1}(-n_1)\cdots E_{m,m+1}(n_k)E_{m,m+1}(-n_k)v=0, \quad \forall n_1,\dots,n_k\in\Z;
\end{equation}
\begin{equation}\label{mm+1}
E_{1,m+n}(m_1)\cdots E_{1,m+n}(m_k)v=0, \quad \forall m_1, \dots, m_k\in\Z.
\end{equation}
\end{lemma}
\begin{proof} Since \eqref{mm+1-} was proven already, we only need to consider  \eqref{mm+1}.
Assume that \eqref{mm+1-} holds for some $k$.
If $k=1$, by applying  $(E_{1 m}(s)E_{m+1,m+n}(s))^2$ to \eqref{mm+1-}, we have
$$E_{1,m+n}(2s+n_1)E_{1,m+n}(2s-n_1)v=0.$$
By setting $p=2s+n_1,\,q=2s-n_1$, we obtain
$$E_{1,m+n}(p)E_{1,m+n}(q)v=0\quad \mbox{ for all }p,q\mbox{ with } p\equiv q~({\rm mod}~2).$$
Since two of any three integers $p,q,a$ must have the same parity, we have
\[
E_{1,m+n}(p)E_{1,m+n}(q)E_{1,m+n}(a)v=0\quad \mbox{ for all }p,q,a\in\Z,
\]
by noting that the elements $E_{1,m+n}(r)$ anti-commute for all $r$.

For $k>1$, applying $(E_{1m}(s)E_{m+1,m+n}(s))^{2k}$ to \eqref{mm+1-}, we have
 $$E_{1,m+1}(2s+n_1)E_{1,m+1}(2s-n_1)\cdots E_{1,m+n}(2s+n_k)E_{1,m+n}(2s-n_k)v=0.$$
Set $p_i=2s+n_i$ and $q_i=2s-n_i$ for $i=1.\dots,k$. Then for all $p_i,q_i\in\Z$ with $p_i\equiv q_i ~({\rm mod}~2)$, we have
$$E_{1,m+n}(p_1)E_{1,m+n}(q_1)\cdots E_{1,m+n}(p_k)E_{1,m+n}(q_k)v=0.$$
Therefore,  for all $p_i,q_i,a_i\in\Z$,
$$E_{1,m+n}(p_1)E_{1,m+n}(q_1)E_{1,m+n}(a_1)\cdots E_{1,m+n}(p_k)E_{1,m+n}(q_k)E_{1,m+n}(a_k)v=0.$$
This proves \eqref{mm+1}.
\end{proof}

By using the lemma, we can prove the following result.
\begin{proposition}\label{prop:key}
Let $V$ be an irreducible zero-level integrable $G$-module with finite dimensional weight spaces.
Then there always exists a nonzero weight vector $w\in V$ such that
\begin{eqnarray}
\label{semi-positive}
&&G^+_{\bar{0}}w=0,\\
\label{semi-positive1}
&&E_{m-i,m+j}(r)w=0, \  (i,j)\ne (0,1), \  r\in\Z, \\
\label{semi-positive2}
&&E_{m,m+1}(n_1)\cdots E_{m,m+1}(n_k)w=0 \quad \text{for large $k$, for all $n_i\in\Z$}.
\end{eqnarray}
\end{proposition}

\begin{proof}
Let $v$ be the vector in Lemma \ref{lem:v-vector} with
$l$ bing the minimal value of $k$  such that the equation \eqref{mm+1} holds. Then there exist $r_1,\dots,r_{l-1}\in\Z$ such that
\begin{eqnarray*}
&&v_{1,m+n}:=E_{1,m+n}(r_1)\cdots E_{1,m+n}(r_{l-1})v\ne 0,\\
&& E_{1,m+n}(r) v_{1,m+n}=0, \quad \forall r\in \Z.
                  \end{eqnarray*}
Since for any $Y\in G_{\bar0}^+$, we have $[Y, E_{1,m+n}(r) ]=0$ for all $r$, and hence
\[
G^+_{\bar{0}}v_{1,m+n}=0.
\]
We observe that \eqref{mm+1-} still holds if we replace $v$ by $v_{1,m+n}$,   namely,  for large $k$ and for all $n_1,\dots,n_k\in\Z$,
\begin{equation}\label{1m+n}
E_{m,m+1}(n_1)E_{m,m+1}(-n_1)\cdots E_{m,m+1}(n_k)E_{m,m+1}(-n_k)v_{1,m+n}=0.
\end{equation}
Applying $(E_{1 m}(s)E_{m+1,m+n-1}(s))^{2k},\,s\in\Z$, to \eqref{1m+n} and using the same arguments for the proof of \eqref{mm+1},
we obtain
\begin{equation}\label{1m+n-1}
E_{1,m+n-1}(m_1)\cdots E_{1,m+n-1}(m_k)v_{1,m+n}=0\quad \mbox{ for all } m_1,\dots,m_k\in\Z.
\end{equation}
Let $l'$ be the minimal integer such that \eqref{1m+n-1} holds. Then there exist $r_1',\dots,r_{l'-1}'\in\Z$ such that
\begin{eqnarray*}&&v_{1,m+n-1}:=E_{1,m+n-1}(r_1')\cdots E_{1,m+n-1}(r_{l'-1}')v_{1,m+n}\ne 0,\\
 &&E_{1,m+n-1}(r)v_{1,m+n-1}=0, \quad \forall r\in\Z.
 \end{eqnarray*}
Therefore,
$G^+_{\bar{0}}v_{1,m+n-1}=0.$
Repeating the above arguments for a finite number of times, we will find a weight vector $w$ such that
\begin{equation}\label{semi-positive0}
G^+_{\bar{0}}w=0,\quad E_{i,m+j}(r)w=0, \,(i,j)\ne (m,1),\,  r\in\Z.
\end{equation}

Let $\mu$ be the weight of $w$.  Observe that $V$, being irreducible, must be cyclically generated by $w$ over $G$.
By using the PBW theorem for the universal enveloping algebra of $G$ and equation \eqref{semi-positive0},
we easily show that any weight of $V$ which is bigger than $\mu$ must be of the form
\begin{eqnarray}\label{eq:is-a-weight}
\mu +a(\es_m-\d_1)+b\d, \quad a\in\Z_{\ge 0},\,b\in\Z.
\end{eqnarray}

Now we prove \eqref{semi-positive2}. Suppose it is false, that is,
for any positive integer $p$,
there always exist $k>p$ and $n_1, \dots, n_k\in \Z$ such that  $\tilde{w}:=E_{m,m+1}(n_1)\cdots E_{m,m+1}(n_k)w\ne 0$.
Then $\nu:=\mu+k(\es_m-\d_1)+\sum\limits_{i=1}^kn_i\d$ is the weight of
$\tilde{w}$. But for large $p$, and hence large $k$, we have $(\nu,\, \es_{m-1}-\es_m)<0$. Thus $\nu+(\es_{m-1}-\es_m)$ is a weight of $V$ by considering the action of the $\mathfrak{sl}(2)$ subalgebra generated by the root spaces  $\fg_\alpha$ and  $\fg_{-\alpha}$ where $\alpha=\es_{m-1}-\es_m$.
However, the weight $\nu+(\es_{m-1}-\es_m)$ is not of the form \eqref{eq:is-a-weight}. This proves \eqref{semi-positive2} by contradiction.
\end{proof}

The following theorem is an easy consequence of Proposition \ref{prop:key}.
\begin{theorem}\label{h-vector}\label{thm:hw}
Let $V$ be an irreducible zero-level integrable $G$-module with finite dimensional weight spaces.
Then there exists a weight vector $v\in V$ such that $G^+ v=0$. Furthermore,
$V$ is isomorphic to the irreducible quotient $V(\varphi)$ of the induced module defined by
\eqref{eq:Verma} for some $\varphi$.
\end{theorem}
\begin{proof} Assume that $\fg=\mathfrak{sl}(m, n)$.
Consider the weight vector $w$ of Proposition \ref{prop:key},
and let $s$ be the minimal integer such that \eqref{semi-positive2} holds.
Then there exist $r_1,\dots,r_{s-1}\in\Z$ such that
\begin{eqnarray*}
&&v:=E_{m,m+1}(r_1)\cdots E_{m,m+1}(r_{s-1})w\ne 0,\\
&& E_{m,m+1}(r) v=0, \quad \forall r\in \Z.
\end{eqnarray*}
It follows from \eqref{semi-positive} and \eqref{semi-positive1} in
Proposition \ref{prop:key}  that $G^+v=0.$

The existence of a highest weight vector can be proved similarly in the case of $C(m)$.
We omit the details, but point out that the following property of $C(m)$ plays a crucial role:
if $\alpha$ and $\beta$ are positive odd roots, then
$[\fg_\alpha, \fg_\beta]=\{0\}=[\fg_{-\alpha}, \fg_{-\beta}]$.
This property is shared by $\mathfrak{sl}(m, n)$.

The last statement of the theorem will be verified in the proof of Theorem \ref{thm:main}. See
in particular remarks below equation \eqref{varphi}.
\end{proof}

\section{Integrable modules for $\hat{A}(m,n)$  and $\hat{C}(m)$}\label{sect:new}
We use $\fg$ to denote $A(m,n)$ with $m\ne n$ or $C(m)$.
Recall that $\fg_{\bar{0}}$ is reductive but not semi-simple. Let $\fg_{ss}$ be the semi-simple part of $\fg_{\bar{0}}$
and $\C z$ be the one dimensional center of $\fg_{\bar{0}}$. Then
$
\fg_{\bar{0}}=\fg_{ss}\oplus \C z.
$
Let $\fg_{ss}=\fg_{ss}^-\oplus \mathfrak{h}_{ss}\oplus \fg_{ss}^+$ be the standard triangular decomposition.
Set
\[
L(\fg_{ss})=\fg_{ss}\otimes L, \quad T^{ss}_0=\mathfrak{h}_{ss}\otimes L\oplus \C c, \quad T^{ss}=T^{ss}_0 \oplus \C d.
\]

Fix a $K$-tuple of integral dominant $\fg_{ss}$-weights $\underline{\l}=(\l_1,\dots,\l_K)$,
and take any $\underline{a}=(a_1,\dots ,a_K)\in \C^K$ with distinct nonzero entries.
Let $\tilde\varphi: \U(T^{ss}_0)\rightarrow L$ and $\psi:\U(T^{ss}_0)\rightarrow \C$  be
algebra homomorphisms respectively defined by
\begin{eqnarray}\label{eq:phi-ss}
\begin{aligned}
\tilde\varphi: &\quad  h\otimes t^s \rightarrow(\sum a_j^s\l_j(h))t^s,\quad c\to 0, \quad h\in\mathfrak{h}_{ss}, \\
\psi: &\quad  h\otimes t^s\rightarrow \sum a_j^s\l_j(h),\quad c\to 0,\quad h\in\mathfrak{h}_{ss},
\end{aligned}
\end{eqnarray}
where $\tilde\varphi$ is $\Z$-graded.
Then $im(\tilde\varphi)\cong L_r$ for some nonnegative integer $r$, and $L_r$ is an irreducible $\U(T^{ss}_0)$-module.
By letting $d$ act on $im(\tilde\varphi)$ by $d t^{i r}=(i r+b)t^{i r}$ ($i\in\Z$) for any fixed $b\in\C$,
we make $im(\tilde\varphi)$ into an irreducible $T^{ss}$-module, which we denote by $L_\varphi$, where
$\varphi=(\tilde\varphi, b)$.

Extend $L_\varphi$ to a module over $B_{ss}:=T^{ss}\oplus \fg_{ss}^+\otimes L$
with $\fg_{ss}^+\otimes L$ acting trivially, and construct the following induced module for $L(\fg_{ss})\oplus \C c\oplus\C d$:
\[
M^0(\varphi)=\U(L(\fg_{ss})\oplus \C c\oplus \C d)\otimes_{\U(B_{ss})}L_\varphi.
\]
This has a unique irreducible quotient, which we denote by $V^0(\varphi)$. Then $V^0(\varphi)$ is an irreducible
loop module by Theorem \ref{thm:ev}.

Let $\U(T^{ss}_0)$ act on the one dimensional vector
space $\C_\psi=\C$ by $\psi$.  We extend $\C_\psi$ to a module over $B_{ss}':=T^{ss}_0\oplus \fg_{ss}^+\otimes L$
with $\fg_{ss}^+\otimes L$ acting trivially, and construct the induced module
\[
M^0(\psi)=\U(L(\fg_{ss})\oplus \C c)\otimes_{\U(B_{ss}')}\C_\psi
\]
for $L(\fg_{ss})\oplus \C c$. This module has a unique simple quotient,
which we denote by $V^0(\psi)$.
We note in particular that $V^0(\psi)$ is integrable with finite dimensional weight spaces.

Let  $I'$ be the ideal  of $L$ generated by
$P'(t)=\prod\limits_{j=1}^K(t-a_j)$. Then it follows from Remark \ref{rem:ideal} that
$(\fg_{ss}\otimes I')V^0(\psi)=0.$

Given any positive integers $b_j$ ($1\le j\le K$),  we let $P(t)=\prod\limits_{j=1}^K(t-a_j)^{b_j}$
and set $\theta=\sum_{i=1}^K b_i$.
Denote by $I$ the ideal generated by $P(t)$, which clearly is contained in $I'$.
We write $P(t)=\sum_{i=0}^\theta c_i t^i$, where the $c_i$ are complex numbers
determined by $a_j$ and $b_j$.  Let $\tau$ be any element of the set 
\begin{eqnarray}\label{eq:tau}
\mathcal{T}_I:=\{\tau\in(z\otimes L)^*\mid \tau(z\otimes I)=0\}.
\end{eqnarray}
Set 
$\tau_s=\tau(z\otimes t^s)$ for all $s\in\Z$,  and we shall also denote
$\tau$ by the sequence $(\tau_i)_{i\in\Z}$. 
We have $\sum_ {i=0}^\theta c_i\tau_{i+m}=0$ for all $m\in\Z.$
Conversely, we may regard this as a linear difference equation of order 
$\theta$ for $(\tau_i)_{i\in\Z}$. It uniquely determines 
the sequence, and hence an element $\tau\in\mathcal{T}_I$, 
by fixing, e.g., $\tau_0, \tau_1, \dots, \tau_{\theta-1}$, to any complex numbers.

One can extend $\psi$ to an algebra homomorphism $\U(T_0)\rightarrow \C$ by letting
$\psi(z\otimes t^s)=\tau_s$ for all $s\in \Z.$
Define the action of $z\otimes L$ on $V^0(\psi)$ by
\[
(z\otimes t^s)u=\tau_su,\quad  s\in \Z,\, u\in V^0(\psi).
\]
Since $[z\otimes L, L(\fg_{ss})\oplus \C c]=0$, this makes
$V^0(\psi)$ into a simple module for  $\fg_{\bar{0}}\otimes L\oplus \C c$.
It follows from \eqref{eq:tau} that $(z\otimes I)V^0(\psi)=0$.
Extend $V^0(\psi)$ to a simple module for the parabolic subalgebra 
$\hat{\mathfrak{p}}:=\fg_{\bar 0}\otimes L\oplus\C c\oplus \mathfrak{n}_{\bar 1}^{+}\otimes L$ 
by letting $\mathfrak{n}_{\bar 1}^{+}\otimes L$ act on $V^0(\psi)$ trivially, and denote the 
resulting module by $V^0(\psi, \tau)$.
Now construct the induced module for $G'=L(\fg)\oplus \C c$:
$$M(\psi,\tau)=\U(G')\otimes_{\U( \hat{\mathfrak{p}})}V^0(\psi, \tau).$$
Standard arguments show that
$M(\psi,\tau)$ has a unique quotient $V(\psi,\tau)$, which is irreducible over $G'$.

\begin{proposition} \label{prop:finite-weight}
We have $(\fg\otimes I)V(\psi,\tau)=0$. This in particular implies that $V(\psi,\tau)$
has finite dimensional weight spaces.
\end{proposition}
\begin{proof} For any positive odd root $\a$, we denote by $x_\a$ and $y_\a$ the root vectors of $\fg$ corresponding to $\a$
and $-\a$ respectively.  Then for all positive odd roots $\a, \b$, we have
$[x_\a\otimes I, y_\b\otimes I]\subset \mathfrak{g}_{\bar0}\otimes I$.
Since $(\mathfrak{g}_{\bar0}\otimes I)V^0(\psi, \tau)=0$, it follows that $(y_\b\otimes I)V^0(\psi, \tau)=0$ for all odd positive roots
$\b$.  Note that $V(\psi,\tau)$ is spanned by
\[
\{y_{\a_1}(n_1)\cdots y_{\a_k}(n_k)u \, |\, \a_i \text{ are odd roots}, \, n_i\in\Z, \, u\in V^0(\psi, \tau), \, k\ge 0\}.
\]
The Lie superalgebra $\fg$ is type I,  thus $[\mathfrak{n}_{\bar 1}^{-}, \mathfrak{n}_{\bar 1}^{-}]=0$,
and it immediately follows that  $(\mathfrak{n}_{\bar 1}^{-}\otimes I)V(\psi,\tau)=0.$

Now $(\fg_{\bar 0}\otimes I)\cdot (y_{\a_1}(n_1)u)=\{0\}$ and
$(\mathfrak{n}_{\bar 1}^{+}\otimes I)\cdot (y_{\a_1}(n_1)u)=0$ for any positive odd root $\a_1$.
By using induction on $k$, it is not difficult to show
that  $(\fg_{\bar 0}\otimes I)\cdot (y_{\a_1}(n_1)\cdots y_{\a_k}(n_k)u)=0$ and
$(\mathfrak{n}_{\bar 1}^{+}\otimes I)\cdot (y_{\a_1}(n_1)\cdots y_{\a_k}(n_k)u)=0$
for all positive odd roots $\a_1, \dots, \a_k$. Hence
$(\fg_{\bar 0}\otimes I)V(\psi,\tau)=(\mathfrak{n}_{\bar 1}^{+}\otimes I)V(\psi,\tau)=0$. This proves that $(\fg\otimes I)V(\psi,\tau)=0$.

Since $I$ is co-finite, $\U(\mathfrak{n}_{\bar1}^-\otimes L/I)=\wedge(\mathfrak{n}_{\bar1}^-\otimes L/I)$ is finite dimensional. This immediately leads to the second statement.
\end{proof}

\begin{remark}\label{rem:int}
Since $V^0(\psi, \tau)$ is an integrable module for $L(\fg_{ss})\oplus \C c$, the induced module $M(\psi,\tau)$ is integrable
with respect to $G'$, and so is also $V(\psi,\tau)$.
\end{remark}

\begin{proposition}\label{prop:psi}The $G'$-modules $V(\psi,\tau)$ and $V(\psi',\tau')$ are isomorphic if and only if
$\tau=\tau'$,  and
there exists a permutation $\sigma$ of $\{1,2,\dots,K\}$ such that $\underline{\l'}=\sigma(\underline{\l})$ and $\underline{a'}=\sigma(\underline{a})$.
\end{proposition}
\begin{proof} The given conditions are necessary and sufficient in order for one to have an isomorphism
 $V^0(\psi, \tau)\cong V^0(\psi', \tau')$ of simple $\mathfrak{p}$-modules. Since  $V^0(\psi, \tau)$
and $V^0(\psi', \tau')$ respectively determine $V^0(\psi, \tau)$  and $V^0(\psi', \tau')$ uniquely, the proposition
follows.
%
\end{proof}

\begin{lemma}\label{lem:ev-condition}
The $V(\psi,\tau)$ is an evaluation $L(\fg)$-module if and only if $\psi(z\otimes I')=0$.
\end{lemma}
\begin{proof} If $\psi(z\otimes I')=0$, then $\fg\otimes I'$ acts on $V(\psi, \tau)$ by zero.  It follows from equation \eqref{eq:zeta} and Remark \ref{rem:ideal} that  $V(\psi,\tau)$ is an evaluation module.  Given any simple evaluation module $V(\underline{\l}, \underline{a})$ (we may and will assume that the entries of $\underline{\l}$ are all nonzero), we have $\psi(z(m))=\sum_{i=1}^K a_i^m \lambda_i(z)$ for all $m\in\Z$ by \eqref{eq:zeta}. Simplicity of $V(\underline{\l}, \underline{a})$ requires that the entries of $\underline{a}=(a_1, \dots, a_K)$ are all distinct.
It is clear that $\psi(z\otimes I')=0$ for $I'=\prod_{i=1}^K(t-a_i)$.
\end{proof}

We turn $V(\psi,\tau)\otimes L$ into a $G$-module using \eqref{L-module}.
Then by Theorem \ref{iso} we have the following  $G$-module isomorphism
\[
V(\psi,\tau)\otimes L\cong \oplus_{i=0}^{r-1}\U(G)w(i),
\]
 where $w$ is a highest weight vector of $V(\psi,\tau)$ and $w(i)=w\otimes t^i$ with $i\in \Z$.
 Note that $\U(G)w(i)$ are irreducible $G$-submodules.

\begin{definition}\label{def:V}
We denote by $\hat{V}(\varphi,\tau)$ the  irreducible $G$-module  $\U(G)w(0)$.
\end{definition}

Here ${\hat V}(\varphi,\tau)$ is a highest weight module, where  $w(0)$ is a highest weight vector.
We have
\begin{eqnarray}
\begin{aligned}
&c w(0)=0, \quad  d w(0)= b w(0), \\
&(z\otimes t^s) w(0)= \tau_s w(s), \\
&(h\otimes t^s) w(0) = \sum a_j^s\l_j(h) w(s ), \quad h\in \fh_{ss}
\end{aligned}
\end{eqnarray}
for all $s\in \Z$, where the $\tau_s$ satisfy \eqref{eq:tau}.
This induces a graded algebra homomorphism $\U(T_0)/\U(T_0)c\longrightarrow L$ since $c$ acts by zero.
The image of this map is  $L_r$ for some integer $r\ge 0$ by \cite[\S 4]{ch86}.  When $r>1$,
we necessarily have $\tau_s=0$ unless $r|s$.

\begin{lemma}\label{lem:V} Under the given conditions on $\underline{\l}$,
$\underline{a}$ and $\tau$,  the irreducible $G$-module ${\hat V}(\varphi,\tau)$
is integrable with finite dimensional weight spaces.
\end{lemma}
\begin{proof}
From equation \eqref{L-module} one can see that  ${\hat V}(\varphi,\tau)$ is integrable over $G$ if and only if
$V(\psi,\tau)$ is integrable over $G'$. Similarly,
${\hat V}(\varphi,\tau)$ has finite dimensional weight spaces (with respect to $\mathfrak{h}\oplus\C c \oplus \C d$)
if and only if $V(\psi, \tau)$ has finite dimensional weight spaces (with respect to $\mathfrak{h}\oplus\C c)$.
It therefore immediately follows from Remark \ref{rem:int} that ${\hat V}(\varphi,\tau)$ is integrable.
By Proposition \ref{prop:finite-weight}, $V(\psi, \tau)$ has finite dimensional weight spaces,
thus ${\hat V}(\varphi,\tau)$ also has finite dimensional weight spaces.
\end{proof}

\begin{proposition}\label{prop:iso-class}
The modules ${\hat V}(\varphi,\tau)$ and ${\hat V}(\varphi',\tau')$ are isomorphic if and only if
all of the following conditions are satisfied:

(a) $b'\equiv b \ ({\rm mod}\, r)$ when $L_\varphi=L_r$,

(b) for $r>0$, there exists $\kappa\in\C\setminus \{0\}$ and a permutation $\sigma$ of $\{1,2,\dots,K\}$ such that
\begin{eqnarray*}
\underline{\l'}=\sigma(\underline{\l}), \quad \underline{a'}=\kappa\sigma(\underline{a}),\quad
(\tau'_i)_{i\in \Z}=(\kappa^i\tau_i)_{i\in\Z}. \label{eq:tau-tau'}
\end{eqnarray*}
\end{proposition}
\begin{remark} If $\tau\in\mathcal{T}_I$ for the ideal $I=(P(t))$ of $L$ generated by $P(t)=\prod_{j=1}^K(t-a_j)^{b_j}$, 
then $\tau'\in\mathcal{T}_J$, where 
$J=(Q(t))$ with $Q(t)=\prod_{j=1}^K(t-\kappa a_j)^{b_j}$. 
\end{remark}
\begin{proof}[Proof of  Proposition \ref{prop:iso-class}] The $r=0$ case is clear, thus we only need to consider the case $r>0$.

The given conditions imply that the irreducible representation
of $G$ on ${\hat V}(\varphi',\tau')$ is the composition of the representation on ${\hat V}(\varphi,\tau)$ and the algebra automorphism of $L$ given by $t\mapsto \kappa t$. Thus the representations are isomorphic.

To prove the converse, we note that the necessity of condition (a) is clear.
The actions of $\U(T_0)$ on the spaces of highest weight vectors of ${\hat V}(\varphi,\tau)$ and
${\hat V}(\varphi',\tau')$ respectively induce surjective graded algebra homomorphisms $\varphi,  \varphi': \U(T_0) \longrightarrow L_r$ for some $r$. As we mentioned before, this in particular implies that
$\varphi(h(m))=\varphi'(h(m))=0$ for all $h(m)\in L(\fh)$ if $r\notdivides m$.
Let $\bar\varphi: \U(T_0)/Ker(\varphi) \longrightarrow L_r$ and
$\bar\varphi': \U(T_0)/Ker(\varphi') \longrightarrow L_r$ be the corresponding canonical isomorphisms.
Let $\bar{x}=\bar\varphi^{-1}(t^r)$, then $\bar\varphi^{-1}(t^{rm})=\bar{x}^m$ for all $m\in \Z$.  If  $x$ is a representative of $\bar{x}$ in $\U(T_0)$, we have $\varphi'(x)=\rho t^r$ for some $\rho\in\C\backslash\{0\}$.  Furthermore,  given any
$h(r m)\in L(\fh)$, we have some $\eta(h(r m))\in \C$ such that
$h(r m)\equiv \eta(h(r m)) x^m \mod Ker(\varphi)$ for all $m$.  Thus
\[
\varphi(h(rm))=  \eta(h(r m)) t^{mr}, \quad        \varphi'(h(rm))=\rho^m \eta(h(r m)) t^{mr}.
\]
These relations lead to the conditions (b) with $\rho=\kappa^r$.
\end{proof}

\section{Classification theorem}\label{sect:main}
In this section, we classify the  irreducible zero-level integrable modules with finite dimensional weight spaces
for the untwisted affine superalgebras $\hat{A}(m,n)$  ($m\ne n$) and  $\hat{C}(m)$.
The following theorem is the main result.

\begin{theorem} \label{thm:main} Let $G$ be either $\hat{A}(m,n)$  ($m\ne n$) or $\hat{C}(m)$.
Any irreducible zero-level integrable $G$-module  with finite dimensional weight spaces is isomorphic to
${\hat V}(\varphi,\tau)$ (see Definition \ref{def:V}) for some  $\varphi$ and $\tau$.
\end{theorem}
\begin{proof} Let $V$ be an irreducible zero-level integrable $G$-module with finite dimensional weight spaces.
By Theorem \ref{h-vector}, there exists a highest weight vector $v$ with weight $\lambda\in H^*$, that is,
$G^+ v=0$ and $h v=\lambda(h) v$  for all $h\in H$.

\begin{claim} 
Let $M=\U(T_0)v$. Then $M$ is an irreducible $T_0$-module.
\end{claim}

Let $w_1,w_2\in M$ be two weight vectors. Then by the irreducibility of $V$, there exists $g\in G$ such that $gw_1=w_2.$
Write $g=\sum_ig_i^- h_i g_i^+$, where $g_i^-\in \U(G^-),\,h_i\in \U(T_0)$, $g_i^+\in \U(G^+).$
Note that $G^+ w_i=0.$ Hence $w_2=\sum_ig_i^- h_i w_1$, which forces all $g_i^-$ to be scalars by weight considerations.
Hence $w_1=h w_2$ for some $h\in \U(T_0)$, that is,  $M$ is an irreducible $T_0$-module.
Observe that $M$ is also an irreducible $T$-module.

\medskip

Since $c$ acts as
zero on $M$ it follows that there exists a maximal graded ideal $\mathcal{N}$ of
$\mathfrak{S}=\U(T_0)/\U(T_0)c$ such that $M\cong \mathfrak{S}/\mathcal{N}$ as
$T_0$-modules. It is  known from \cite{ch86} that
$M\cong \mathfrak{S}/\mathcal{N}\cong L_r$ for some integer $r\ge 0$.
Let $\tilde\varphi$ be the natural  map defined by the following
composition
\begin{equation}\label{varphi}
\tilde\varphi:\U(T_0)\xrightarrow[]{\mbox{action}} M\cong \mathfrak{S}/\mathcal{N}\cong
L_r\subset L.
\end{equation}
Clearly, $\tilde\varphi$ is $\Z$-graded, $\tilde\varphi(c)=0$ and $\tilde\varphi|_\mathfrak{h}\in \mathfrak{h}^*$.
Moreover  $V$ is isomorphic to $\hat{V}(\varphi)$, the irreducible quotient of the induced module defined by \eqref{eq:Verma}
with $\varphi=(\tilde\varphi, \lambda(d))$.
If $r=0$, the module is trivial. Thus we shall assume $r>0$ hereafter.

As in Section 2, we construct a simple module $V(\psi)$ for $G'=L(\fg)\oplus \C c$ from $V\cong \hat{V}(\varphi)$
by setting $\psi=S\circ \tilde\varphi$. Since $V$ has finite dimensional weight spaces,
it follows from Theorem \ref{iso}
that $V(\psi)$ has finite dimensional weight spaces and there exists a co-finite ideal $I$ of $L$ such that
\begin{equation}
(\fg\otimes I)V(\psi)=0.
\end{equation}

This ideal can be determined as follows (for more detail see \cite[Lemma 3.7]{r04}).
Let $w$ be a highest weight vector of $V(\psi)$ and let $\mu$ be its weight.
For  each simple root $\a\in \Pi$, we let $y_\a$ be a root vector for the root $-\a$. Consider
$
\{y_\a(s) w|s\in\Z\},
$
which is contained in the same weight space $V(\psi)_{\mu-\a}$. Since $\dim V(\psi)_{\mu-\a}<\infty$,
there exists a non-zero polynomial $P_{\a}(t)$ such that $(y_\a\otimes P_\a(t))w=0.$
Set $P(t)=\prod_{\a\in\Pi}P_\a(t)$. Then $I$ is the ideal generated by $P(t)$.

To avoid confusion, we change the notation of $V(\psi)$ to $V_{ev}$.
Recall  that $\fg_{ss}$ is the semi-simple part of $\fg_{\bar{0}}$. Regard $V_{ev}$ as a module for $L(\fg_{ss})\oplus \C c$,
and set $V^0(\psi)=\U(L(\fg_{ss})\oplus\C c)w$.

\begin{claim} 
$V^0(\psi)$ is an irreducible module for $L(\fg_{ss})\oplus \C c$.
\end{claim}
Recall the standard triangular decomposition $\fg_{ss}=\fg_{ss}^-\oplus \mathfrak{h}_{ss}\oplus \fg_{ss}^+$ for $\fg_{ss}$.
Since $(\fg^+_{ss}\otimes L)w=0$ and $[\fg^+_{ss}\otimes L,T]\subset \fg^+_{ss}\otimes L$, we have
$
V^0(\psi)=\U(\fg_{ss}^-\otimes L)w.
$
Let $u\in V^0(\psi)$ be a weight vector.  So we can write $u=\sum_ig_i h_i w$ for some $g_i\in \U(\fg_{ss}^-\otimes L)$
and $h_i\in \U(\mathfrak{h}_{ss}\otimes L)$.  Since $V_{ev}$ is an irreducible $G'$-module,
there exists $x\in \U(G^+)$ such that $x(\sum_i g_i h_i w)=w$.   Weight considerations require
$x\in \U(\fg^+_{ss}\otimes L).$  This proves that $V^0(\psi)$ is an irreducible module for $L(\fg_{ss})\oplus \C c$.

Since $V(\psi)$ has finite dimensional $\fh$-weight spaces, so does also $V^0(\psi)$.
As $z$ is in the center of $\fg_{\bar0}+\C c$, and acts on $w$ by the multiplication by the scalar $\psi(z)$, 
it acts on the entire $V^0(\psi)$ by the multiplication by $\psi(z)$. 
It follows that $V^0(\psi)$ has finite dimensional $\fh_{ss}$-weight spaces.
Therefore, there exists a co-finite ideal $J$ of $L$ such that $(\fg_{ss}\otimes J)V^0(\psi)=0$, where
$J$ is generated by  $Q(t):=\prod_{\a\in \Delta^+_{\bar 0}\cap\Pi}P_\a(t)$.

 By remark \ref{ideal} we  can  assume that $Q(t)$ has non-zero roots.  The module $V^0(\psi)$ is trivial
 when $Q(t)$ is a constant. Thus we will further assume that $Q(t)$ is not a constant.
 Then up to a scalar multiple, $Q(t)$ has the unique factorisation  $Q(t)=\prod\limits_{i=1}^S(t-a_i)^{s_i}$,
 where $a_i$ are distinct non-zero complex numbers, $s_i$ and $S$ are positive integers.
 Let $Q'(t)=\prod\limits_{i=1}^S(t-a_i)$ and let  $J'$ be the ideal generated by $Q'(t)$.
 Similar arguments as those  in \cite[Proposition 5.2]{rz04} show that
$
(\fg_{ss}\otimes J')V^0(\psi)=0.
$

Therefore $V^0(\psi)$  is a $\fg_{ss}\otimes L /J'$-module.  Because of the particular form of the generator $Q'(t)$ of $J'$,
 we have $\fg_{ss}\otimes L /J'\cong \underbrace{\fg_{ss}\oplus \cdots \oplus \fg_{ss}}_S$ by Remark \ref{rem:ideal}
. Since $V^0(\psi)$
 is an irreducible integrable module, it is finite dimensional. Clearly
any finite dimensional irreducible module for $\underbrace{\fg_{ss}\oplus \cdots \oplus \fg_{ss}}_S$
is isomorphic to $V(\l_1)\otimes  \cdots \otimes V(\l_S)$ for some $\l_i\in \fh^*_{ss}$,
which are integral dominant with respect to $\mathfrak{g}_{ss}$.
Here $ V(\l_i)$ denotes an irreducible $\fg_{ss}$-module with highest weight $\lambda_i$.
Therefore, $V^0(\psi)$ is isomorphic to $V(\l_1)\otimes  \cdots \otimes V(\l_S)$ as
an irreducible module for $L(\fg_{ss})\oplus \C c$ via
the map \eqref{eq:zeta}.

Again by Remark \ref{ideal} we can assume that $P(t)$ has non-zero roots. We can also assume that $P(t)$ is not a constant.
Since $P(t)$ has $Q(t)$ as a factor, it factorises into
\[
P(t)=\prod\limits_{i=1}^S(t-a_i)^{b_i}\prod\limits_{j=S+1}^K(t-a_j)^{b_j}
 \]
for some $K\ge S$. Here all the $a_i$ are distinct nonzero complex numbers,
and $b_i\ge s_i$ if $1\le i\le S$.

 Let $I$ be the ideal of $L$ generated by $P(t)$. Set $P'(t)=\prod\limits_{i=1}^K(t-a_i),$
 and let $I'$ be the ideal of $L$ generated by $P'(t)$. Clearly, $I'\subset J'$.
 For $S+1\le j\le K$, let $V(\l_j)=V(0)=\C$ be the one-dimensional $\fg_{ss}$-module. We have
\[
V^0(\psi)\cong V(\l_1)\otimes \cdots \otimes V(\l_K)
\]
with the action given by equation \eqref{eq:zeta}. This is also an
isomorphism of $\fg_{ss}\otimes L/I'$-modules.

Let $\tau\in (z\otimes L)^*$ be defined by $\tau(x)=\psi(x)$ for all $x\in z\otimes L$. Then 
$\tau\in \mathcal{T}_I$ (see \eqref{eq:tau}). Note that 
$\tau(z\otimes t^s)=\psi(z\otimes t^s)=0$ if $r\notdivides  s$ since $\varphi(z\otimes t^s)=0$ if $r\notdivides  s$. 
Now 
$z\otimes t^s$ acts on the highest weight
vector $w$ by $(z\otimes t^s) w=\tau(z\otimes t^s)w$. Since $[z\otimes L,L(\fg_{ss})\oplus \C c]=0$, for all $u\in V^0(\psi)$,
\[
(z\otimes t^s)u=\tau(z\otimes t^s)u,\quad s\in\Z.
\]
This makes $V^0(\psi)$ into an irreducible $\fg_{\bar 0}\otimes L\oplus \C c$-module. 

As $(\mathfrak{n}_{\bar 1}^+\otimes L) w=0$ and
$[\mathfrak{n}_{\bar 1}^+\otimes L, L(\fg_{ss})\oplus \C c]\subset  \mathfrak{n}_{\bar 1}^+\otimes L$, we have
 $(\mathfrak{n}_{\bar 1}^+\otimes L) V^0(\psi)=0.$  
Regard $V^0(\psi)$ as a module for $\mathfrak{p}\otimes L/I\oplus \C c$ where $\mathfrak{p}:=\fg_{\bar 0}\oplus \mathfrak{n}_{\bar 1}^+$,   and denote it by $V^0(\psi, \tau)$.
 Construct the induced module
 \[
 M(\psi,\tau)=\U(\fg\otimes L/I\oplus \C c)\otimes_{\U(\mathfrak{p}\otimes L/I\oplus \C c)} V^0(\psi, \tau).
 \]
Then the unique irreducible quotient  module $V(\psi,\tau)$ of $M(\psi,\tau)$ is isomorphic to $V_{ev}$.
It then follows from part (1) of Theorem \ref{iso} that $V\cong {\hat V}(\varphi, \tau)$ for some $\varphi$ and $\tau$ as in Definition \ref{def:V}.
\end{proof}

\begin{remark}\label{rem:final}
From Lemma \ref{lem:ev-condition} we see that not all $\hat{V}(\varphi, \tau)$ are loop modules, but the lemma
provides the necessary and sufficient conditions for $\hat{V}(\varphi, \tau)$ to be one.
\end{remark}

\section{Concluding remarks}
Theorem \ref{thm:main} and Theorem \ref{thm:ev}  together classify the irreducible zero-level integrable modules with finite dimensional weight spaces for $\hat{A}(m,n)$ ($m\ne n$) and $\hat{C}(m)$ .
In view of results of \cite{rz04} and \cite{JZ, kw01} discussed in Section \ref{sect:intro}, this completes the
classification of the irreducible integrable modules with finite dimensional
weight spaces for all the untwisted affine superalgebras.
An interesting fact in the case of $\hat{A}(m,n)$ ($m\ne n$) and $\hat{C}(m)$ is that such modules are
are not necessarily loop modules.

\vspace{1cm}

\noindent {\bf Acknowledgement.}  We thank S. Rao for helpful comments on an earlier version of this paper.
This work was supported by the Chinese National Natural Science Foundation grant No. 11271056,
Australian Research Council Discovery-Project grant No. DP0986349, Qing Lan Project of Jiangsu
Province, and Jiangsu Overseas Research and Training Program for Prominent Young and Middle Aged University Teachers and Presidents.   This paper was completed while
both authors were visiting the University of Science and Technology of China, Hefei.

\end{document}